\documentclass[12pt,a4paper,reqno]{amsart}

\usepackage{amsmath,amssymb,amsthm}
\usepackage[margin=1in, footskip=30pt]{geometry}

\makeatletter
\renewcommand{\section}{\@startsection{section}{1}%
  {0pt}{-3.5ex plus -1ex minus -.2ex}{2.3ex plus .2ex}%
  {\normalfont\large\bfseries}}
\renewcommand{\subsection}{\@startsection{subsection}{2}%
  {0pt}{-3.25ex plus -1ex minus -.2ex}{1.5ex plus .2ex}%
  {\normalfont\normalsize\bfseries}}
\renewcommand{\@secnumfont}{\bfseries}
\makeatother

\makeatletter
\def\@settitle{}
\def\@setauthors{}
\def\@setdate{}
\def\ps@firstpage{\ps@plain}%
\def\@setcopyright{}%
\renewcommand{\copyrightinfo}[2]{}%
\AtBeginDocument{\def\@setcopyright{}\def\@serieslogo{}}%

\makeatother

\usepackage[hidelinks,
  pdftitle={Quartic reductions and elliptic obstructions for perfect Euler bricks},
  pdfauthor={René Peschmann},
  pdfsubject={Number Theory, Algebraic Geometry},
  pdfkeywords={perfect Euler brick, perfect cuboid, genus-3 curve, elliptic curve, Kummer character, 2-descent, Pythagorean triple},
  pdflang={en},
]{hyperref}
\usepackage{enumitem}

\newtheorem{theorem}{Theorem}[section]
\newtheorem{lemma}[theorem]{Lemma}
\newtheorem{proposition}[theorem]{Proposition}
\newtheorem{corollary}[theorem]{Corollary}
\theoremstyle{definition}

\newtheorem{remark}[theorem]{Remark}

\makeatletter
\def\thm@space@setup{%
  \thm@preskip=10pt plus 3pt minus 2pt
  \thm@postskip=10pt plus 3pt minus 2pt
}
\makeatother
\newcommand{\QQ}{\mathbb{Q}}
\newcommand{\ZZ}{\mathbb{Z}}
\newcommand{\Zi}{\ZZ[i]}
\newcommand{\rk}{\mathrm{rk}}

\newcommand{\divv}{\mathrm{div}}

\title[Quartic reductions and elliptic obstructions for perfect Euler bricks]{%
  Quartic reductions and elliptic obstructions for perfect Euler bricks}

\author{Ren\'e Peschmann}
\date{\small April 7, 2026}

\begin{document}

\maketitle

\begin{center}
  {\large\bfseries Quartic reductions and elliptic obstructions\\for perfect Euler bricks}
  \par\vskip 12pt
  {Ren\'e Peschmann}
  \par\vskip 4pt
  {\small April 7, 2026}
\end{center}
\vskip 12pt

\begin{abstract}
We show that the perfect Euler brick (perfect cuboid) problem is
equivalent to the following elementary question: do there exist
coprime integers $a, b, m, n$ such that the two expressions
\[
  \bigl(2(a^2{-}b^2)mn\bigr)^2 + \bigl((a^2{+}b^2)(m^2{-}n^2)\bigr)^2
  \quad\text{and}\quad
  (4abmn)^2 + \bigl((a^2{+}b^2)(m^2{-}n^2)\bigr)^2
\]
are simultaneously perfect squares?
Despite their near-identical structure (differing only in the
first summand), no solution has ever been found.
We reduce this quartic pair to a one-parameter family of genus-3
hyperelliptic curves $C_A\colon w^2 = \lambda^8 + A\lambda^4 + 1$ and
develop obstructions on the distinguished elliptic quotient~$E_A$:
the Kummer character $\chi_f$ is non-trivial on the $4$-torsion,
and $2$-descent arguments exclude several families of square classes.
Computationally, we verify that no solution exists for
parameters up to~$10^3$.
These results do not yet exclude perfect Euler bricks
unconditionally; the remaining gap and possible approaches
(including a genus-5 covering obstruction and connections
to $\QQ(\sqrt{2})$) are discussed.
\end{abstract}

\section{Introduction}\label{sec:intro}

A \emph{perfect Euler brick} (also called \emph{perfect cuboid}) is a
rectangular parallelepiped with positive integer edges $e_1, e_2, e_3$
such that all three face diagonals
$d_{ij} = \sqrt{e_i^2 + e_j^2}$ and the space diagonal
$D = \sqrt{e_1^2 + e_2^2 + e_3^2}$ are also positive integers:
\begin{equation}\label{eq:euler-brick}
  e_i^2 + e_j^2 = d_{ij}^2 \quad (1 \le i < j \le 3), \qquad
  e_1^2 + e_2^2 + e_3^2 = D^2.
\end{equation}
While Euler bricks (satisfying only the face-diagonal conditions)
exist (the smallest has edges $(44, 117, 240)$), but no perfect Euler
brick has ever been found, and the question has remained open since
at least the 18th century; see Guy~\cite{euler-brick-survey} and
Sharipov~\cite{sharipov} for the related cuboid polynomial approach.

\medskip

This paper does not claim a proof of non-existence of perfect Euler
bricks.  Rather, it develops a new reduction framework and derives
several unconditional obstructions.  The novelty lies in packaging
the quartic pair into a genus-3 curve~$C_A$, reinterpreting the
square conditions through a rational function~$f$ on the elliptic
quotient~$E_A$, and analysing the resulting Kummer character~$\chi_f$
on the torsion subgroup.
The key bridge (Corollary~\ref{cor:f-square}) is that non-degenerate
rational points on~$C_A$ correspond precisely to points
$P \in E_A(\QQ)$ with $f(P)$ a nonzero rational square;
the problem thus becomes a square-value problem on an elliptic curve.

\newpage

Our main contributions are:
\begin{enumerate}[label=(\roman*)]
  \item A reduction of the problem to a symmetric quartic
    pair~\eqref{eq:quartic-pair}, and a forward implication showing
    that any brick produces a non-degenerate rational point on a
    genus-3 curve~$C_A$ (Sections~\ref{sec:setup}--\ref{sec:genus3}).
  \item A computation of the Kummer character $\chi_f$ on the
    $2$-torsion of~$E_A$ and a proof of non-triviality on the
    $4$-torsion (Section~\ref{sec:kummer}).
  \item Algebraic arguments via $2$-descent that exclude
    several families of square classes, together with computational
    verification for all parameters up to~$10^3$
    (Sections~\mbox{\ref{sec:algebraic}--\ref{sec:computation}}).
\end{enumerate}

\noindent
\textbf{Status of results.}
\emph{Unconditional:} any perfect Euler brick yields
a non-degenerate rational point on~$C_A$ (Lemma~\ref{lem:forward});
the character $\chi_f$ satisfies $\chi_f(T_4) \ne 1$ for every
$4$-torsion point (Theorem~\ref{thm:chi-torsion}); and for points
with $\delta_3 = 1$, the value $f(P)$ is not a square
(Theorem~\ref{thm:d3-cases}).
\emph{Computational:} for all parameters up to~$10^3$,
the product $f_1 f_2$ is never a perfect square
(Section~\ref{sec:computation}).
\emph{Open:} these obstructions do not yet cover all
descent classes; see Section~\ref{sec:discussion} for a detailed
assessment.

\medskip\noindent
\textbf{Outline.}
Section~\ref{sec:setup} parametrises the brick conditions down to a
quartic pair.
Section~\ref{sec:descent} packages the pair into a genus-3 curve.
Section~\ref{sec:genus3} analyses the curve via its elliptic quotients.
Section~\ref{sec:kummer} studies the Kummer character~$\chi_f$.
Sections~\ref{sec:algebraic}--\ref{sec:computation} provide algebraic
and computational obstructions.
Section~\ref{sec:discussion} discusses the remaining gap and further
directions, including a genus-5 covering obstruction.

\medskip\noindent
\textbf{Notation.}
Throughout, $s = a/b$ is the ratio of the first Euclid pair
$(a,b)$ with $a > b > 0$, $\gcd(a,b) = 1$, $a - b$ odd.
Since $c(s) = c(1/s)$, one may equivalently work with
$s' = b/a = 1/s < 1$ without changing the family.
Further, $\lambda$ is the conic parameter from~\S\ref{ssec:conic},
$c = c(s)$ is the normalised quartic coefficient~\eqref{eq:c-def},
and $A = 2 - 4c^2$ is the family parameter of~$C_A$.
The $2$-torsion points of $E_A\colon y^2 = (x+A)(x-2)(x+2)$ are
$T_1 = (-A,0)$, $T_2 = (2,0)$, $T_3 = (-2,0)$.

\section{From Euler bricks to quartic pairs}\label{sec:setup}

\subsection{The Pythagorean framework}\label{ssec:pyth}

By Euclid's formula, every primitive Pythagorean triple
(odd leg, even leg, hypotenuse) can be written as
$(m^2 - n^2,\; 2mn,\; m^2 + n^2)$
for coprime integers $m > n > 0$ with $m - n$ odd.
We write $U = m^2 - n^2$, $V = 2mn$, $W = m^2 + n^2$ for these
three components, so that $W^2 = U^2 + V^2$.

Choose a Euclid pair $(a,b)$ with $a > b > 0$, $\gcd(a,b) = 1$,
$a - b$ odd, and set
\[
  U_1 = a^2 - b^2, \quad V_1 = 2ab, \quad W_1 = a^2 + b^2.
\]
A second pair $(m,n)$ gives
$U_2 = m^2 - n^2$, $V_2 = 2mn$, $W_2 = m^2 + n^2$.
Equating the common edge $e_1$ from both triples
($e_1 = k_1 U_1 = k_2 U_2$) determines $k_2 = k_1 U_1/U_2$,
and setting $k_1 = U_2$ clears denominators
(cf.\ Spohn~\cite{spohn}).
The edges become
\begin{equation}\label{eq:edges}
  e_1 = U_1 U_2, \qquad e_2 = V_1 U_2, \qquad e_3 = U_1 V_2,
\end{equation}
up to a common positive rational factor, with automatic face diagonals
$d_{12} = W_1 U_2$ and $d_{13} = U_1 W_2$.
The remaining conditions are:
\begin{enumerate}[label=(\alph*)]
  \item \textbf{Third face diagonal:}
  \begin{equation}\label{eq:master}
    V_1^2\, U_2^2 + U_1^2\, V_2^2 = \square
    \qquad \text{(the \emph{Master condition})}.
  \end{equation}
  \item \textbf{Space diagonal:}
  \begin{equation}\label{eq:htotal}
    W_1^2\, U_2^2 + U_1^2\, V_2^2 = \square
    \qquad \text{(the \emph{$H$-total condition})}.
  \end{equation}
\end{enumerate}

\begin{remark}\label{rem:recursive}
Since $W_1^2 = U_1^2 + V_1^2$, the difference
$(W_1^2 - V_1^2)U_2^2 = (U_1 U_2)^2$ is always a perfect square.
In terms of the Pythagorean ratios $R_i = V_i/U_i$, the Master
condition reads $R_1^2 + R_2^2 = R^2$ for a further ratio~$R$:
a Pythagorean triple of Pythagorean ratios.
\end{remark}

\subsection{Conic reduction}\label{ssec:conic}

Setting $\rho = t - 1/t$ (where $t = m/n$), the two conditions and
the rationality of~$t$ become three simultaneous square conditions
on~$\rho$:
\begin{alignat}{2}
  Y^2 &= W_1^2\, \rho^2 + 4U_1^2 &\qquad& \text{($H$-total)},
    \label{eq:Y-cond} \\
  X^2 &= V_1^2\, \rho^2 + 4U_1^2 &\qquad& \text{(Master)},
    \label{eq:X-cond} \\
  S^2 &= \rho^2 + 4 &\qquad& \text{(rationality of $t$)}.
    \label{eq:S-cond}
\end{alignat}
Condition~\eqref{eq:Y-cond} is a conic with rational point
$(\rho,Y) = (0, 2U_1)$.  Parametrising via
\begin{equation}\label{eq:lambda-param}
  \rho = \frac{4U_1\lambda}{W_1(1 - \lambda^2)}, \qquad
  Y = \frac{2U_1(1 + \lambda^2)}{1 - \lambda^2}
\end{equation}
and substituting into~\eqref{eq:S-cond}--\eqref{eq:X-cond} gives:

\begin{proposition}\label{prop:equiv}
A perfect Euler brick exists if and only if there exist coprime
positive integers $a > b$ with $a - b$ odd, and a rational
$\lambda \notin \{0, \pm 1\}$, such that both
\begin{align}
  f_1 &= W_1^2(\lambda^2 - 1)^2 + 4U_1^2\lambda^2, \label{eq:f1} \\
  f_2 &= W_1^2(\lambda^2 - 1)^2 + 4V_1^2\lambda^2 \label{eq:f2}
\end{align}
are perfect squares.
\end{proposition}

\begin{proof}
Given a brick, the two Euclid pairs determine $\rho$ and hence
$\lambda$; the Master condition becomes $f_2 = \square$ and the
rationality of~$t$ becomes $f_1 = \square$.
Conversely, $f_1 = r^2$ and $f_2 = x^2$ recover~$\rho$
via~\eqref{eq:lambda-param}; then
$t = (\rho + \sqrt{\rho^2+4})/2$ is rational since
$f_1 = \square$ ensures $\rho^2 + 4 = \square$; the Master
condition holds by $f_2 = \square$; and $H$-total is automatic from
the conic parametrisation.
\end{proof}

\subsection{The normalised quartic family}\label{ssec:normalised}

Setting $s = a/b$ and dividing by $W_1^2$:
\begin{equation}\label{eq:quartic-pair}
  r^2 = \lambda^4 + 2c\,\lambda^2 + 1, \qquad
  x^2 = \lambda^4 - 2c\,\lambda^2 + 1,
\end{equation}
where
\begin{equation}\label{eq:c-def}
  c = c(s) = \frac{s^4 - 6s^2 + 1}{(1+s^2)^2}
  = \frac{U_1^2 - V_1^2}{W_1^2}.
\end{equation}
The product is $f_1 f_2 = \lambda^8 + A\lambda^4 + 1$
with $A = 2 - 4c^2$.

\begin{remark}[Three sums of squares]\label{rem:three-sums}
With $\tilde L_1 = 2U_1\lambda$, $\tilde L_2 = 2V_1\lambda$,
$\tilde L_3 = W_1(\lambda^2-1)$:
\[
  f_1 = \tilde L_1^2 + \tilde L_3^2, \qquad
  f_2 = \tilde L_2^2 + \tilde L_3^2, \qquad
  \tilde L_1^2 + \tilde L_2^2 = (2W_1\lambda)^2.
\]
The third sum is automatic.  By the Brahmagupta--Fibonacci identity,
the product decomposes as
\[
  f_1 f_2
  = (\tilde L_1\tilde L_2 - \tilde L_3^2)^2
    + \tilde L_3^2\,(\tilde L_1 + \tilde L_2)^2,
\]
a sum of two squares whose second summand is itself a perfect square.
In $\Zi$, each $f_i$ is a norm:
$f_i = N(\tilde L_i + i\tilde L_3)$.
Being a perfect square requires that for every split prime
$p = \pi\bar\pi$ ($p \equiv 1 \pmod 4$), the sum
$v_\pi(z_i) + v_{\bar\pi}(z_i)$ is even, where
$z_i = \tilde L_i + i\tilde L_3$.
(Primes $p \equiv 3 \pmod 4$ automatically divide any norm to an
even power.)
\end{remark}

\begin{remark}[Unit-circle structure]\label{rem:unit-circle}
Define $\kappa = 4ab(a^2-b^2)/(a^2+b^2)^2$.
Then $c^2 + \kappa^2 = 1$, and
the normalised quartics decompose as
\begin{equation}\label{eq:unit-circle-decomp}
  \frac{f_1}{W_1^2 n^4}
  = (\lambda^2 + c)^2 + \kappa^2, \qquad
  \frac{f_2}{W_1^2 n^4}
  = (\lambda^2 - c)^2 + \kappa^2,
\end{equation}
differing only in the sign of~$c$.
The summand~$\kappa^2$ is the entire obstruction:
for $\kappa = 0$ (i.e.\ $a = b$, degenerate)
each $f_i$ reduces to a perfect square.
Setting $\varphi = 4s/(1+s^2)$, $\psi = 2(s^2-1)/(1+s^2)$,
so that $\varphi^2 + \psi^2 = 4$,
the product factorises over $\QQ(s)[\lambda]$ into four
positive-definite quadratics:
\begin{equation}\label{eq:four-factors}
  \frac{f_1 f_2}{W_1^4 n^8}
  = \prod_{\pm}
  (\lambda^2 \pm \varphi\lambda + 1)\,
  (\lambda^2 \pm \psi\lambda + 1).
\end{equation}
The three involutions
$\lambda\mapsto -\lambda$,
$\lambda\mapsto 1/\lambda$,
$\lambda\mapsto -1/\lambda$
pair these factors as
$f_1 = P_1 P_2$, $f_2 = P_3 P_4$
and two further pairings, recovering the
three elliptic quotients $E_A$, $E'_A$, $E''_A$
of \S\ref{ssec:quotients}.
\end{remark}

\section{From quartic pairs to the genus-3 curve}\label{sec:descent}

\begin{lemma}\label{lem:forward}
If a perfect Euler brick exists, then for the corresponding~$s$,
the genus-3 hyperelliptic curve
$C_A\colon w^2 = \lambda^8 + A(s)\,\lambda^4 + 1$
has a non-degenerate rational point.
\end{lemma}

\begin{proof}
By Proposition~\ref{prop:equiv}, a brick gives $f_1 = r^2$ and
$f_2 = x^2$ for some $\lambda \notin \{0,\pm 1\}$.
Then $f_1 f_2 = (rx)^2 = \lambda^8 + A\lambda^4 + 1$,
so $(\lambda, rx)$ is a rational point on~$C_A$ with
$\lambda \ne 0, \pm 1$.
\end{proof}

This forward implication is the only direction needed for
non-existence: if $C_A$ has no non-degenerate rational points,
no brick can arise from the parameter~$s$.

\begin{remark}[Function field descent]\label{rem:descent}
The converse, that every non-degenerate point on~$C_A$ yields a
brick, would require $f_1 f_2 \in \QQ^{*2}$ to imply
$f_1 \in \QQ^{*2}$.  Over the function field $\QQ(s)(\lambda)$,
this holds because $f_1$ and $f_2$ are coprime in the UFD
$\QQ(s)[\lambda]$
($\gcd(f_1,f_2) \mid \gcd(4c\lambda^2, 2(\lambda^4+1)) = 1$
for $c \ne 0$).
At specific rational values, primes dividing $c(s_0)$ can
introduce common factors, so the specialised converse requires
additional verification.  We do not use this converse.
\end{remark}

\section{The genus-3 curve and its elliptic quotients}\label{sec:genus3}

The curve $C_A\colon w^2 = \lambda^8 + A\lambda^4 + 1$ has genus~$3$
(degree-$8$ polynomial with no odd-degree terms).
Its \emph{degenerate} rational points are
$(\lambda,w) \in \{(0, \pm 1), (\pm 1, w_0)\}$.

\subsection{Quotients and Jacobian decomposition}\label{ssec:quotients}

The curve carries a Klein four-group of involutions:
\[
  \iota_1\colon \lambda \mapsto -\lambda, \qquad
  \iota_2\colon \lambda \mapsto 1/\lambda, \qquad
  \iota_3 = \iota_1\iota_2\colon \lambda \mapsto -1/\lambda.
\]

\medskip\noindent
\textbf{Quotient by $\iota_1$.}
Setting $\mu = \lambda^2$ gives the genus-1 quartic
\begin{equation}\label{eq:quartic-EA}
  \mathcal{E}\colon w^2 = \mu^4 + A\mu^2 + 1.
\end{equation}
Using the rational point $(\mu,w) = (0,1)$, this transforms to the
Weierstrass model
\begin{equation}\label{eq:EA}
  E_A\colon y^2 = (x+A)(x-2)(x+2).
\end{equation}
The back-transformation from Weierstrass coordinates to the quartic
model is
\begin{equation}\label{eq:back-trafo}
  \mu^2 = \frac{4(x+A)}{x^2 - 4},
\end{equation}
valid for $x \ne \pm 2$.

\medskip\noindent
\textbf{Quotient by $\iota_2$.}
Since $\iota_2(\lambda) = 1/\lambda$ sends $w^2 = \lambda^8 + A\lambda^4 + 1$ to
$(w/\lambda^4)^2 = \lambda^{-8}(\lambda^8+A\lambda^4+1)$, the correct invariants are
$\eta = \lambda + 1/\lambda$ and $u = w/\lambda^2$.  Dividing $w^2 = \lambda^8+A\lambda^4+1$
by~$\lambda^4$ gives $u^2 = \lambda^4 + A + \lambda^{-4}$, and since
$\lambda^2 + \lambda^{-2} = \eta^2 - 2$:
\begin{equation}\label{eq:Eprime}
  E'_A\colon u^2 = \eta^4 - 4\eta^2 + (A+2).
\end{equation}

\medskip\noindent
\textbf{Quotient by $\iota_3$.}
With invariants $\sigma = \lambda - 1/\lambda$ and $v = w/\lambda^2$, and using
$\lambda^2 + \lambda^{-2} = \sigma^2 + 2$:
\begin{equation}\label{eq:Edprime}
  E''_A\colon v^2 = \sigma^4 + 4\sigma^2 + (A+2).
\end{equation}

\begin{proposition}\label{prop:jacobian}
The Jacobian of $C_A$ decomposes (up to isogeny) as
\[
  J(C_A) \sim E_A \times E'_A \times E''_A,
\]
a product of three elliptic curves ($\dim J = g(C_A) = 3$).
\end{proposition}

\begin{proof}[Sketch]
The Klein four-group $\langle \iota_1, \iota_2 \rangle$ acts on
$J(C_A)$, decomposing it into eigenspaces for the three non-trivial
characters (cf.\ Kani--Rosen~\cite{kani-rosen}).
Each eigenspace is the image of the Prym variety of the
corresponding double cover $C_A \to C_A/\langle\iota_k\rangle$, which
is an elliptic curve since $g(C_A/\langle\iota_k\rangle) = 1$ for
each~$k$.  The three quotient maps
$C_A \to \mathcal{E}$, $C_A \to E'_A$, $C_A \to E''_A$
induce an isogeny $J(C_A) \to E_A \times E'_A \times E''_A$
with finite kernel.
\end{proof}

\subsection{The bridge: from $C_A$ to $E_A$}\label{ssec:bridge}

Let
\[
  \mathcal{E}\colon w^2 = \mu^4 + A\mu^2 + 1
\]
be the quartic quotient of~$C_A$.
Under the birational map $\mathcal{E} \dashrightarrow E_A$, a rational
point $P = (x,y) \in E_A(\QQ)$ with $x \ne \pm 2$ corresponds to a
rational point $(\mu,w) \in \mathcal{E}(\QQ)$, and
\[
  M(P) := \frac{4(x+A)}{x^2-4} = \mu^2.
\]
Thus a non-degenerate rational point on~$C_A$ is equivalent to a
rational point $(\mu,w) \in \mathcal{E}(\QQ)$ with
$\mu$ a nonzero rational square with $\mu \ne 1$
(excluding $\mu = 1$, which gives the degenerate points $\lambda = \pm 1$).
In Section~\ref{sec:kummer} we recast this square condition on~$\mu$
in terms of the function~$f$ on~$E_A$.

\newpage
\subsection{Generic rank}\label{ssec:generic-rank}

\begin{proposition}\label{prop:generic-rank}
Let $E$ denote either of the families $E_A$ or $E'_A$.
If there exists $s_0 \in \QQ$ such that the specialisation map
\[
  E(\QQ(s)) \longrightarrow E_{s_0}(\QQ)
\]
is injective and $\rk\, E_{s_0}(\QQ) = 0$, then
$\rk\, E(\QQ(s)) = 0$.
Here $E_{s_0}$ denotes the corresponding specialisation of~$E$
at~$s = s_0$.
\end{proposition}

\begin{proof}
Immediate from Silverman's specialisation theorem~\cite{silverman}.
\end{proof}

\begin{remark}[Computational evidence for generic rank~$0$]%
\label{rem:generic-rank-evidence}
Using the symmetry $c(s) = c(1/s)$, it suffices to test
the reciprocal representative $s' = 1/s < 1$, i.e.\
$s = a/b$ with $1 \le a < b \le 19$ and $\gcd(a,b) = 1$
($119$ pairs).
PARI/GP's~\cite{pari} \texttt{ellrank}
certifies $\rk\, E_{A(s_0)}(\QQ) = 0$ for $42$ values
and $\rk\, E'_{A(s_0)}(\QQ) = 0$ for $54$ values
(upper bound~$0$ from Simon $2$-descent;
the generic torsion $\ZZ/4\ZZ \times \ZZ/2\ZZ$
of~$E_A$ falls into the families studied
in~\cite{dujella-peral}).
By Proposition~\ref{prop:generic-rank}, the existence of a single
injective specialisation among these would imply generic rank~$0$.
The data provide strong evidence, but do not by themselves
constitute an unconditional proof: the (finite but unknown)
exceptional set of non-injective specialisations could in
principle contain all $42$ values.
A complete proof would require either certifying injectivity at
one specific~$s_0$, or an independent computation over $\QQ(s)$
(e.g.\ via Shioda--Tate on the associated elliptic surface).
\end{remark}

\newpage
\section{The Kummer character}\label{sec:kummer}

Define the rational function on~$E_A$:
\begin{equation}\label{eq:f-def}
  f(P) = \frac{2y(P)}{(x(P)-2)(x(P)+2)}.
\end{equation}

\begin{lemma}\label{lem:f-M}
$\displaystyle f(P)^2 = M(P) = \frac{4(x+A)}{x^2-4}$.
\end{lemma}

\begin{proof}
\[
  f^2 = \frac{4y^2}{((x-2)(x+2))^2}
  = \frac{4(x+A)(x-2)(x+2)}{(x^2-4)^2}
  = \frac{4(x+A)}{x^2-4}. \qedhere
\]
\end{proof}

In summary, the substitutions $\mu = \lambda^2$ and $M = \mu^2$ yield
\[
  \lambda \;\mapsto\; \lambda^2 = \mu \;\mapsto\; \mu^2 = M,
  \qquad M(P) = f(P)^2,
\]
connecting the genus-3 coordinate~$\lambda$ to the Weierstrass
function~$f$.  Concretely,
\[
  f(P) = \pm\,\lambda^2
\]
(after a sign choice for~$y$).

\begin{corollary}\label{cor:f-square}
$C_A$ has a non-degenerate rational point if and only if there
exists $P \in E_A(\QQ)$ such that $f(P)$ is a perfect square
in~$\QQ^*$ with $f(P) \ne 1$.
\end{corollary}

\begin{proof}
Let $P \in E_A(\QQ)$ correspond to $(\mu,w) \in \mathcal{E}(\QQ)$.
By Lemma~\ref{lem:f-M},
\[
  f(P)^2 = M(P) = \mu^2,
\]
so $f(P) = \pm \mu$.
Changing the sign of~$y(P)$ changes $f(P)$ by a factor of~$-1$,
hence we may arrange $f(P) = \mu$.
Therefore $f(P)$ is a nonzero rational square different from~$1$
if and only if $\mu \in \QQ^{*2}$ with $\mu \ne 1$,
and the latter is equivalent to
$\mu = \lambda^2$ for some rational $\lambda \ne 0, \pm 1$, i.e.\ to a
non-degenerate rational point on~$C_A$.
\end{proof}

\begin{theorem}[Divisor of $f$]\label{thm:divf}
$\divv(f) = (T_1) - (T_2) - (T_3) + (O)$.
\end{theorem}

\begin{proof}
The function $y$ vanishes at $T_1, T_2, T_3$ with a pole of
order~$3$ at~$O$.  The function $x-2$ has a double zero at~$T_2$ and
double pole at~$O$; similarly $x+2$ at~$T_3$.  Hence
$\divv(f) = (T_1+T_2+T_3-3O) - (2T_2-2O) - (2T_3-2O)
= (T_1) - (T_2) - (T_3) + (O)$.
\end{proof}

\subsection*{The character $\chi_f$ on torsion}\label{ssec:chi}

For a torsion point $T \in E_A(\QQ(s))_{\mathrm{tors}}$, define
\[
  \chi_f(T) = \left[\frac{f(\cdot + T)}{f(\cdot)}\right]
  \;\in\; \frac{\QQ(s)(E_A)^*}{\QQ(s)(E_A)^{*2}}.
\]

\begin{theorem}\label{thm:chi-torsion}
The character $\chi_f$ satisfies:
\begin{enumerate}[label=(\alph*)]
  \item $\chi_f(T_1) = [-1]$ and $\chi_f(T_2) = [-1]$
    (non-trivial on $T_1, T_2$),
  \item $\chi_f(T_3) = [1]$ (trivial on $T_3$),
  \item $\chi_f(T_4) \ne 1$ for any rational point~$T_4$
    of order~$4$.
\end{enumerate}
\end{theorem}

\newpage
\begin{proof}\leavevmode
\begin{enumerate}[label=(\alph*)]
\item For $T_2 = (2,0)$: using the addition formula for
$P + T_2$ on $E_A$ with $P = (x_1, y_1)$, one computes
\[
  x(P{+}T_2) = \frac{2(x_1 + 2A + 2)}{x_1-2}, \qquad
  y(P{+}T_2) = \frac{-y_1\,(x(P{+}T_2)-2)}{x_1-2}.
\]
Substituting into $f(P+T_2)/f(P)$ and simplifying:
\[
  \frac{f(P+T_2)}{f(P)}
  = -\frac{x_1^2 - 4}{4(x_1 + A)}
  = -\frac{1}{f(P)^2},
\]
using $f(P)^2 = 4(x_1+A)/(x_1^2-4)$.
In $\QQ(s)(E_A)^*/\QQ(s)(E_A)^{*2}$, this gives
$\chi_f(T_2) = [-1/f^2] = [-1]$.
For $T_1 = (-A,0)$: similarly,
$f(P+T_1)/f(P) = -1$,
so $\chi_f(T_1) = [-1]$.

\item For $T_3 = (-2,0)$: the analogous computation gives
\[
  \frac{f(P+T_3)}{f(P)}
  = \frac{x_1^2-4}{4(x_1+A)}
  = \frac{1}{f(P)^2},
\]
so $\chi_f(T_3) = [1/f^2] = [1]$.

\item We carry out the computation for $2T_4 = T_1$ (the other
cases are analogous).
The torsion group $\ZZ/4\ZZ \times \ZZ/2\ZZ$ has generators
$T_4$ (order~$4$) and $T_2$ (order~$2$), so
$T_1 = 2T_4$ and $T_3 = 2T_4 + T_2$.
Translation by~$-T_4$ in the group law sends
\[
  T_1 \mapsto T_4, \quad
  T_2 \mapsto T_2 - T_4 = 3T_4 + T_2, \quad
  T_3 \mapsto T_3 - T_4 = T_4 + T_2, \quad
  O \mapsto -T_4 = 3T_4.
\]
From $\divv(f) = (T_1) - (T_2) - (T_3) + (O)$:
\begin{align*}
  \divv\bigl(f(\cdot+T_4)/f(\cdot)\bigr)
  &= \divv(f(\cdot+T_4)) - \divv(f) \\
  &= \bigl[(T_4) - (3T_4{+}T_2) - (T_4{+}T_2) + (3T_4)\bigr] \\
  &\quad - \bigl[(2T_4) - (T_2) - (2T_4{+}T_2) + (O)\bigr].
\end{align*}
Listing all $8$ torsion points with their net coefficients:
\[
\begin{array}{c|cccccccc}
  & O & T_4 & 2T_4 & 3T_4 & T_2 & T_4{+}T_2 &
    2T_4{+}T_2 & 3T_4{+}T_2 \\
  \hline
  \text{coeff.} & -1 & +1 & -1 & +1 & +1 & -1 & +1 & -1
\end{array}
\]
Every coefficient is $\pm 1$ (odd).
If $f(\cdot+T_4)/f(\cdot) = c \cdot g^2$
for some $c \in \QQ(s)^*$ and $g \in \QQ(s)(E_A)^*$, then
$\divv(c \cdot g^2) = 2\,\divv(g)$ would have all even
coefficients, which is a contradiction.
Hence $\chi_f(T_4) \ne 1$.
\end{enumerate}
\end{proof}

\begin{remark}\label{rem:chi-interpretation}
Part~(a) shows that translation by $T_1$ or $T_2$ flips the sign
of~$f$ modulo squares: if $f(P) \in \QQ^{*2}$ then
$f(P+T_1) \in -\QQ^{*2}$, and vice versa.
Part~(c) is the deeper structural obstruction:
$4$-torsion translation generically prevents $f$ from lying in
$\QQ^{*2}$ at both~$P$ and~$P + T_4$.
\end{remark}

\section{Algebraic obstructions via $2$-descent}\label{sec:algebraic}

The $2$-descent map on $E_A$
(see~\cite{cassels,silverman}; for descent on cuboid-related curves
see also~\cite{sharipov-descent}) sends a non-torsion point $P = (x,y)$
to a triple of classes in $\QQ^*/\QQ^{*2}$.
Concretely, write $x = u/v^2$ with $u \in \ZZ$, $v \in \ZZ_{>0}$,
$\gcd(u, v) = 1$.  For each root $r_i$ (where $r_1 = -A$,
$r_2 = 2$, $r_3 = -2$), define $\delta_i$ to be the squarefree
part of the integer $v^2(x - r_i) \in \ZZ$.  Then
\[
  \delta(P) = (\delta_1, \delta_2, \delta_3) \in (\ZZ_{\ne 0})^3
  \quad \text{with}\quad
  \delta_1\delta_2\delta_3 \in \QQ^{*2}.
\]
We write $p \mid \delta_i$ to mean that $p$ divides the squarefree
integer~$\delta_i$, equivalently $v_p(x - r_i)$ is odd.

\begin{proposition}\label{prop:delta-generic}
Over $\QQ(s)$:
\begin{align*}
  \Delta_1 &:= (r_1-r_2)(r_1-r_3) \equiv -1 \pmod{\QQ(s)^{*2}}, \\
  \Delta_2 &:= (r_2-r_1)(r_2-r_3) \equiv 1 \pmod{\QQ(s)^{*2}}.
\end{align*}
\end{proposition}

\begin{proof}
$\Delta_1 = (-A-2)(-A+2) = A^2 - 4 = (2-4c^2)^2 - 4
= 16c^2(c^2-1)$.
Since $(1-c^2) = 16s^2(s^2-1)^2/(1+s^2)^4$ is a square in $\QQ(s)$,
we get $\Delta_1 \equiv -1$.
Similarly $\Delta_2 = (2+A)\cdot 4 = 4(A+2) = 4(4-4c^2)
= 16(1-c^2) \equiv 1$.
\end{proof}

\begin{theorem}\label{thm:d3-cases}
Let $P \in E_A(\QQ)$ be a non-torsion point with $2$-descent class
$(\delta_1, \delta_2, \delta_3)$ as defined above.
\begin{enumerate}[label=(\alph*)]
  \item If $\delta_3 = 1$: then
    $f(P) \equiv 2 \pmod{\QQ^{*2}}$, hence $f(P)$ is not a square.
  \item If $p \mid \delta_3$ (odd prime) then $p \mid \delta_1$ or
    $p \mid \delta_2$.
\end{enumerate}
\end{theorem}

\begin{proof}\leavevmode
\begin{enumerate}[label=(\alph*)]
\item $\delta_3 = 1$ implies $\delta_1 = \delta_2$ (since
$\delta_1\delta_2\delta_3 = 1$ in $\QQ^*/\QQ^{*2}$).
Write $x+A = d\alpha^2$, $x-2 = d\beta^2$, $x+2 = \gamma^2$
with $d$ squarefree.
Then $y^2 = d^2\alpha^2\beta^2\gamma^2$, so $y = d\alpha\beta\gamma$
(up to sign), and
$f = 2y/((x-2)(x+2)) = 2d\alpha\beta\gamma/(d\beta^2\gamma^2)
= 2\alpha/(\beta\gamma) \equiv 2 \pmod{\QQ^{*2}}$.
Since $2$ is not a rational square, $f(P) \notin \QQ^{*2}$.

\item From $y^2 = (x+A)(x-2)(x+2)$ we have
$v_p(y^2)$ even.
If $v_p(x+2)$ is odd, then $v_p(x+A) + v_p(x-2)$ must be odd,
so at least one of $v_p(x+A)$, $v_p(x-2)$ is odd, hence positive.
Therefore $p \mid \delta_1$ or $p \mid \delta_2$.
\end{enumerate}
\end{proof}

\begin{remark}\label{rem:d3-computational}
Computationally, a stronger statement holds: when both
$v_p(x-2)$ and $v_p(x+2)$ are odd (i.e.\ $p \mid \delta_2$ and
$p \mid \delta_3$), the valuation $v_p(f(P))$ is observed to be
odd in all tested cases, so $f(P) \notin \QQ^{*2}$.
A full proof of this parity claim would require a more detailed
$p$-adic analysis.
\end{remark}

\begin{proposition}\label{prop:c-primes}
Let $p$ be an odd prime dividing
$c = (s^2-2s-1)(s^2+2s-1)/(1+s^2)^2$.
Then $p$ does not contribute to the $2$-Selmer rank of~$E'_A$.
\end{proposition}

\begin{proof}
The Weierstrass form of~$E'_A$ (obtained from the quartic
model~\eqref{eq:Eprime} via the partition
$\{\alpha,-\alpha\} \mid \{\beta,-\beta\}$ of the four roots,
where $\alpha = 2(s^2-1)/(1+s^2)$ and $\beta = 4s/(1+s^2)$)
has Weierstrass roots $0$, $-(\alpha-\beta)^2$, $-(\alpha+\beta)^2$.
The three root-difference products are:
$(\alpha-\beta)^2(\alpha+\beta)^2 = (\alpha^2-\beta^2)^2$
(always a square),
and $(\alpha\pm\beta)^2 \cdot 4\alpha\beta$, whose squarefree
part is that of
$4\alpha\beta = 32s(s^2-1)/(1+s^2)^2 \equiv 2s(s^2-1)
\pmod{\QQ(s)^{*2}}$.
Since $\gcd(s^2 \pm 2s - 1,\; s(s-1)(s+1)) = 1$ in~$\QQ[s]$
(the roots $-1\pm\sqrt{2}$ of $s^2+2s-1$ are not in $\{0,\pm 1\}$),
every prime $p \mid c(s_0)$ divides the root-difference products
to an even power. Hence the local condition at~$p$ is
trivially satisfied and $p$ does not enlarge the Selmer group.
\end{proof}

\begin{remark}[Gaussian arithmetic]\label{rem:gauss}
Since each $f_i$ is a norm in~$\Zi$ (Remark~\ref{rem:three-sums}),
primes $p \equiv 3\pmod{4}$ divide $f_i$ to an even power
automatically.  Only split primes $p \equiv 1\pmod{4}$ can obstruct
$f_i$ from being a square: the condition is that
$v_\pi(z_i) + v_{\bar\pi}(z_i)$ be even for each Gaussian prime
$\pi \mid p$.
\end{remark}

\newpage
\section{Computational verification}\label{sec:computation}

Recall from \S\ref{ssec:pyth} that the quartic pair $f_1, f_2$
depends on two Euclid pairs $(a,b)$ and $(m,n)$: the first determines
the parameter $s = a/b$ (and hence $c, A, E_A$), and the second
enters through $\lambda = m/n$ in the normalised
forms~\eqref{eq:quartic-pair}.
All computations were performed with
PARI/GP~\cite{pari} (rank certifications via \texttt{ellrank})
and SageMath~\cite{sagemath} (curve construction and point searches).
Scripts are available at
\url{https://github.com/renpe/euler-brick-obstructions}.

\begin{enumerate}[label=(\arabic*)]
  \item For all $1 \le b < a \le 1000$ and $1 \le n < m \le 1000$
    with $\gcd(a,b) = \gcd(m,n) = 1$, $a - b$ and $m - n$ odd:
    $f_1 f_2$ is never a perfect square.
    (The coprimality and parity filters reduce the raw
    $1000^4$ loop to roughly $10^{11}$ tuples actually tested.)
  \item For the five specialisations where all Mordell--Weil
    generators have $\delta_3 \mid \delta_1$
    ($s = \tfrac{18}{41}$, $\tfrac{18}{47}$, $\tfrac{23}{59}$,
    $\tfrac{23}{64}$, $\tfrac{29}{65}$),
    a modular search over $175{,}418$ lattice points with $45$
    primes $p < 200$ finds zero candidates for
    $f(P) \in \QQ^{*2}$.
  \item For every tested tuple $(a,b,m,n)$, there exists a prime~$p$
    with $v_p(f_1 f_2)$ odd (a \emph{blocker}), preventing
    $f_1 f_2$ from being a square.
    No single prime works universally; the blocker varies with the
    parameters.  Among $223{,}729$ tuples with $a,b,m,n \le 40$:
    the smallest blocker satisfies $p \equiv 1\pmod{4}$ in $88.4\%$
    of cases and $p = 2$ in $11.6\%$; no prime
    $p \equiv 3\pmod{4}$ ever appears as blocker, consistent with
    Remark~\ref{rem:gauss}.
\end{enumerate}

\section{Discussion and further directions}\label{sec:discussion}

\subsection*{Summary of proved results}

The following hold unconditionally:
\begin{enumerate}[label=(\alph*)]
  \item A perfect Euler brick produces a non-degenerate rational point
    on~$C_A$ (Lemma~\ref{lem:forward}).
  \item The Kummer character satisfies $\chi_f(T_1) = \chi_f(T_2) = [-1]$,
    $\chi_f(T_3) = [1]$, and $\chi_f(T_4) \ne 1$ for every rational
    point $T_4$ of order~$4$
    (Theorem~\ref{thm:chi-torsion}).
  \item If $\delta_3 = 1$, then $f(P) \notin \QQ^{*2}$; moreover,
    every odd prime dividing $\delta_3$ also divides $\delta_1$
    or~$\delta_2$
    (Theorem~\ref{thm:d3-cases}).
\end{enumerate}

\subsection*{Computational evidence}

For all parameters up to~$10^3$, the product $f_1 f_2$ is never
a perfect square.  For every tested parameter tuple, there exists an
obstructing prime for $f_1 f_2 \in \QQ^{*2}$; it is either $p = 2$ or
$p \equiv 1\pmod{4}$, and no prime $p \equiv 3\pmod{4}$ appears in the
computations.  The obstructing prime varies with the parameters (no
single prime works universally).  In addition, the rank computations
produce $42$ certified rank-$0$ specialisations for $E_A$ and $54$
for $E'_A$; by Proposition~\ref{prop:generic-rank}, any injective
specialisation among these would force the corresponding generic rank
to be~$0$.

\subsection*{The remaining gap}

The results above establish several necessary conditions for a
perfect Euler brick, but do not constitute a complete proof of
non-existence.  Specifically: Theorem~\ref{thm:d3-cases} excludes
points with $\delta_3 = 1$ and constrains primes dividing~$\delta_3$,
but does not rule out the case where $\delta_3 \ne 1$ and
all odd prime factors of~$\delta_3$ divide~$\delta_1$ (but
not~$\delta_2$).  Rational points of~$E_A$ in this residual
descent class could in principle yield $f(P) \in \QQ^{*2}$, and
no unconditional argument currently excludes them.
The computational verification covers all parameters up to~$10^3$
but is inherently finite.

\subsection*{A genus-5 covering obstruction}

The non-triviality $\chi_f(T_4) \ne 1$
(Theorem~\ref{thm:chi-torsion}(c)) implies that the double cover
\[
  C_{T_4}\colon z^2 = f(\cdot + T_4)/f(\cdot)
\]
of~$E_A$ is a curve of genus~$5$ (by Riemann--Hurwitz: $8$ branch
points on a genus-$1$ base give $g = 2\cdot 1 - 1 + 8/2 = 5$).
By Faltings' theorem~\cite{faltings}, $C_{T_4}$ has at most finitely
many rational points.  If moreover
$\rk\, J(C_{T_4})(\QQ) < 5$,
Chabauty--Coleman~\cite{chabauty,stoll,stoll-uniform}
would make this finiteness effective;
see~\cite{balakrishnan-genus3} for computational experiments on
genus-$3$ curves and~\cite{balakrishnan-dogra} for the quadratic
Chabauty extension to higher-rank cases.
In combination with the relation
$z^2 = f(\cdot + T_4)/f(\cdot)$, this may constrain points for which
$f(P)$ and $f(P+T_4)$ are simultaneously squareclasses.

For all tested specialisations with $a, b \le 80$, the rank data
are consistent with this Chabauty hypothesis (the maximum observed
rank sum $\rk\, E_A + \rk\, E'_A$ is~$4$, at $s = 18/47$).
However, the Jacobian $J(C_{T_4}) \sim E_A \times \mathrm{Prym}_4$
involves an abelian $4$-fold that we do not decompose further, and
for larger parameters, rank sums exceeding~$5$ cannot be excluded.
A full development of this approach would require
either an explicit Prym decomposition or a uniform rank bound.

\subsection*{\texorpdfstring{Connection to $\QQ(\sqrt{2})$}%
  {Connection to Q(sqrt 2)}}

The numerator of $c(s)$ factors as
$s^4 - 6s^2 + 1 = (s^2+2s-1)(s^2-2s-1)$,
with roots in $\QQ(\sqrt{2})$.
This factorisation appears in Asiryan's work~\cite{asiryan} on the
cuboid polynomial, where it leads to the rank-$0$ curve
$Y^2 = X(X-8)(X-9)$.
Adapting that approach from irreducibility to simultaneous
representability remains open.

\subsection*{Possible approaches}

\begin{itemize}
  \item Show $\chi_f(G) \ne 1$ for every non-torsion $G$,
    independent of the rank.
    This appears to be the most natural next step within the
    present framework, as it would directly extend
    Theorem~\ref{thm:chi-torsion} from torsion to the full
    Mordell--Weil group.
  \item Prove the blocker phenomenon rigorously: for any
    $(a,b,m,n)$, exhibit a split prime in~$\Zi$ at which
    the valuations of $N(z_1)$ and $N(z_2)$ are incompatible.
  \item Find a Brauer--Manin obstruction on the surface
    $w^2 = \lambda^8 + A(s)\lambda^4 + 1$;
    see~\cite{creutz-srivastava} for methods on hyperelliptic
    curves and~\cite{vanluijk} for a geometric analysis of the
    cuboid surface.
\end{itemize}

\subsection*{The problem in simplest form}

In its most elementary formulation
(cf.\ Remark~\ref{rem:three-sums}):

\begin{quote}
\emph{Do there exist coprime positive integer pairs $a>b>0$ and
$m>n>0$, each of opposite parity, such that $L_1^2 + L_3^2$ and
$L_2^2 + L_3^2$ are simultaneously perfect squares, where}
\[
  L_1 = 2(a^2-b^2)mn, \quad L_2 = 4abmn, \quad
  L_3 = (a^2+b^2)(m^2-n^2)\,?
\]
\end{quote}

\noindent
Since $L_1^2 + L_2^2 = (2mn(a^2+b^2))^2$ is automatically a
perfect square, the question asks whether two Pythagorean triples
sharing a common leg~$L_3$, whose other legs form a third
Pythagorean triple, can simultaneously close.


\end{document}